\documentclass[14pt]{amsart}
\address{\newline{\normalsize Moscow Institute of Physics and Technology, 9 Institutskiy per., Dolgoprudny,
 Moscow Region, 141701, Russia}
\bigskip
\newline{\it E-mail address}: karzhemanov.iv@mipt.ru}
\usepackage{amscd,amsthm,amsmath,amssymb}
\usepackage[dvips]{graphicx}
\usepackage[matrix,arrow]{xy}

\usepackage{textcomp}

\makeatletter\@addtoreset{equation}{section}\makeatother

\makeatletter\@addtoreset{subsection}{equation}\makeatother

\newcommand{\p}{\mathbb{P}}
\newcommand{\cel}{\mathbb{Z}}
\newcommand{\na}{\mathbb{N}}
\newcommand{\com}{\mathbb{C}}

\newcommand{\pg}{\mathrm{PGL}}

\newcommand{\rea}{\mathbb{R}}

\newcommand{\map}{\longrightarrow}

\newtheorem{theorem}[equation]{Theorem}
\newtheorem{prop}[equation]{Proposition}
\newtheorem{lemma}[equation]{Lemma}

\theoremstyle{remark}
\newtheorem{remark}[equation]{Remark}

\newtheorem{ex}[equation]{Example}

\newtheorem{definition}[equation]{Definition}

\thanks{{\it MS 2010 classification}: 14E05, 37F10, 14J26}


\begin{document}

\title{On the dynamical degree of surjective endomorphisms}

\author{Ilya Karzhemanov}

\begin{abstract}
We establish a couple of dynamical properties of surjective
rational maps $f: X \dashrightarrow X$ for smooth projective
surfaces $X$. We also give a numerical condition for regularity of
$f$. Some explicit
constructions and calculations, related to the topological entropy
of $f$, are provided.
\end{abstract}

\sloppy

\maketitle

\bigskip

\section{Introduction}
\label{section:f}

\refstepcounter{equation}
\subsection{}
\label{subsection:p-1}

Let $X$ be a smooth projective complex surface. In the paper
\cite{Ilya-dP}, we have initiated the study of the geometry of
surjective endomorphisms $f: X \dashrightarrow X$, that is
rational maps satisfying the property $f(X \setminus Z) = X$ for
some finite subset $Z \subset X$ (cf. \cite{Ilya-MLD},
\cite{Ilya-quant} and \cite{Ilya-ML} for statistical, physical and
computational aspects of the story). The aim of the present note
is to highlight some \emph{dynamical} sides of such maps $f$.

Consider the N\'eron--Severi group $N^1(X)$. The arguments in \cite[{\bf 2.1}]{Ilya-dP} show
that $f$ induces a linear map $f^*: N^1(X) \to N^1(X)$ with the
property
$$
(f^n)^* = (f^*)^n \ \forall \ n \in \na \quad \text{for the
corresponding iterates}.
$$
The maps with this \emph{algebraic stability} behavior had
acquired a great deal of studies in rational dynamics (see
\cite{Diller-Fav}, \cite{Truong}, \cite{Sib}). In this regard, one
important numerical characteristic of $f$ is the \emph{(first)
dynamical degree} $\lambda(f)$, defined as follows. Let $H \in
\text{Pic}(X)$ be an ample divisor. Then
$$
\lambda(f) := \lim_{n \to \infty} \sqrt[n]{((f^n)^*(H) \cdot H)}.
$$
This definition does not actually depend on the choice of $H$.

Here is a (classical) relation between $\lambda(f)$ and the
spectral radius $\rho(f^*)$ of the linear operator $f^*$:

\begin{prop}
\label{theorem:lambda-1-is-rho} $\lambda(f) = \rho(f^*)$ and there
is a (\emph{Perron}) vector $v \in N^1(X)$, which is a nef class
on $X$, satisfying $f^*(v) = \lambda(f)v$.
\end{prop}

This claim is very well-known to specialists of course. Yet we
include its proof for transparency and in order to adapt to our
present setting (see Section~\ref{section:spe} below).

\begin{remark}
\label{remark:other-surfs-1}
Proposition~\ref{theorem:lambda-1-is-rho} implies that
$\lambda(f)$ is an \emph{algebraic integer} because $f^*$
preserves the (full) lattice $\text{Pic}(X)$. However, in the case
of an arbitrary rational dominant map $X \dashrightarrow X$,
already when $X = \p^2$, its dynamical degree can be
\emph{transcendental} (see \cite{Bell-et-al}). Note also that
Proposition~\ref{theorem:lambda-1-is-rho} holds for any surjective
map $f$ and any surface $X$ with irregularity $q(X) = 0$ (see the
arguments in Section~\ref{section:spe}).
\end{remark}

\refstepcounter{equation}
\subsection{}
\label{subsection:p-2}

In Section~\ref{section:pro}, we establish a relation between
$\lambda(f)$ and the (topological) degree $\deg f$ of $f$, namely
that $\lambda(f)^2 \ge \deg f$. Again, this is a classical
\emph{log-concavity property} of the dynamical degrees, which we
prove via algebro-geometric means. Next, following the steps of
the proof, we obtain a \emph{numerical} condition for regularity
(in the current setting):

\begin{theorem}
\label{theorem:main} Let $\lambda(f)^2 = \deg f$ and $f$ is mild
(see Definition~\ref{definition:mild} below). Then the map $f$ is
\emph{regular}.
\end{theorem}

\begin{remark}
\label{remark:other-surfs-2} The case of monomial endomorphisms of
$\p^2$ shows that Theorem~\ref{theorem:main} \emph{does not} hold
for arbitrary rational dominant maps $X \dashrightarrow X$ (see
\cite[Section 3]{Favre}). Note also that regularity of $f$ need
not imply $\lambda(f)^2 = \deg f$: consider the surface $\p^1
\times \p^1$ and the map $f$ given by $(t,w) \mapsto (t^2,w)$.
Finally, even in the surjective case, Theorem~\ref{theorem:main}
may fail for \emph{non-mild} maps $f$, as the following construction
(again for the quadric) shows:\footnote{We thank the anonymous for
pointing out this.}
$$
f: (t,w) \mapsto (t^2,w^2\frac{t-2}{t-4});
$$
this (surjective) map has $\lambda(f) = 2$ and $\deg f = 4$.
\end{remark}

The paper closes with Section~\ref{section:exam}, where we discuss
some basic properties of the entropy of $f$, which we further
support by explicit calculations for a wide class of surjective
endomorphisms. We also consider a seemingly interesting class of
rank $1$ maps (see Definition~\ref{definition:def-rank-1}) and
discuss their structure.

\bigskip

\section{Spectral interpretation (revisited)}
\label{section:spe}

\refstepcounter{equation}
\subsection{}
\label{subsection:sp-1}

Let $V$ be a finite-dimensional real normed space and $K \subset
V$ a closed cone with nonempty interior. Suppose $T: V \map V$ is
a linear map such that $T(K) \subseteq K$. Denote by $\rho(T)$ the
spectral radius of $T$.

Let us recall the following general result (see \cite{K-R}):

\begin{theorem}[Krein--Rutman]
\label{theorem:krein-rutman} In the previous setting, $\rho(T)$ is
an eigenvalue of $T$ and there exists an eigenvector $v \in K, v
\ne 0$, such that $T(v) = \rho(T)v$.
\end{theorem}

\refstepcounter{equation}
\subsection{}
\label{subsection:sp-2}

We retain the notations from Introduction. Take $V := N^1(X)$, $T
:= f^*$ and $K := $ the nef cone of $X$ (cf. \cite[Lemma
2.5]{Ilya-dP}), so that Theorem~\ref{theorem:krein-rutman}
provides a nef $\rea$-class $v$ on $X$ with $f^*(v) = \rho(f^*)v$.

Choose the divisor $H$ from {\ref{subsection:p-1}} above in the
form $v + H'$ for some ample $H' \in N^1(X)$.

\begin{lemma}
\label{theorem:lambda-rho-1} There exists an absolute constant $C$
such that
$$
C^{-1}\rho(f^*)^n \le ((f^*)^n(H) \cdot H) \le C\rho(f^*)^n.
$$
\end{lemma}

\begin{proof}
The upper bound follows from the operator norm growth. In turn,
the lower bound follows from the fact that $v$ is an eigenvector
of $f^*$ and that the intersections $H \cdot v, H \cdot H'$ are
positive.
\end{proof}

In order to prove Proposition~\ref{theorem:lambda-1-is-rho} it
suffices to take the $n^{\text{th}}$ roots of the inequalities in
Lemma~\ref{theorem:lambda-rho-1} and then pass to the limit $n \to
\infty$.

\begin{remark}
\label{remark:nef-not-big} Note that the class $v$ \emph{is not
necessarily big} (i.e. it need not satisfy $(v^2) > 0$). As an
example we consider $X : = \p^1_{\bf t} \times \p^1_{\bf w}$ and
its (regular) endomorphism
$$
f([t_0:t_1] \times [w_0:w_1]) := [t^d_0:t^d_1] \times [w_0:w_1]
$$
for some $d \ge 2$. In the basis of $N^1(X)$, corresponding to two
rulings on $X$, the operator $f^*$ is represented by the matrix
$\begin{pmatrix} d & 0 \\ 0 & 1 \end{pmatrix}$. Then we have
$\rho(f^*) = d$ and $(v^2) = 0$.
\end{remark}

\bigskip

\section{Relation between $\lambda$ and $\deg$}
\label{section:pro}

\refstepcounter{equation}
\subsection{}
\label{subsection:pro-1}

Consider a resolution diagram for $f$:

\begin{equation}
\nonumber \xymatrix{
& Y \ar[dl]_g \ar[dr]^{h} \\
X \ar@{-->}[rr]^{f} && X.}
\end{equation}
Here $g$ is a composition of blow-ups and $h$ is a regular
morphism. Let us give the following:

\begin{definition}
\label{definition:mild} The map $f$ is called \emph{mild} if for an
appropriate $g$ and the Perron class $v$ from
Proposition~\ref{theorem:lambda-1-is-rho} one has $(v \cdot h_*(E)) > 0$ for any $g$-exceptional curve $E$.
\end{definition}

Let us see first that $\lambda(f)^2 \ge \deg f$ for a given surjective $f$. Then for $H$ from {\ref{subsection:p-1}} we have $f^*(H) =
g_*(h^*(H))$ and so
$$
h^*(H) = g^*(f^*(H)) - E,
$$
where $E := \displaystyle\sum_i a_i E_i$, $E_i$ are
$g$-exceptional and $a_i$ are some coefficients. It follows that
$$
(f^*(H)^2) = (h^*(H)^2) - E^2 = (\deg f)(H^2) - E^2 \ge (\deg
f)(H^2)
$$
because the intersection matrix $\{(E_i \cdot E_j)\}_{i,j}$ is
negative definite.

Further, since both classes $f^*(H)$ and $H$ are nef, the Hodge
Index Theorem gives
$$
(f^*(H) \cdot H)^2 \ge (f^*(H)^2)(H^2).
$$
Thus we get
$$
(f^*(H) \cdot H)^2 \ge (\deg f)(H^2)^2.
$$

Finally, replacing $f$ with $f^n$ and using $\deg f^n = (\deg
f)^n$, we get
$$
\sqrt[n]{((f^n)^*(H) \cdot H)^2} \ge (\deg f)\sqrt[n]{(H^2)^2},
$$
which for $n \to \infty$ becomes $\lambda(f)^2 \ge \deg f$, as
claimed.

\refstepcounter{equation}
\subsection{}
\label{subsection:pro-2}

We now turn to the proof of Theorem~\ref{theorem:main}. Let
$\lambda(f)^2 = \deg f$ and suppose the map $f$ is
\emph{non-regular}. The latter means that the resolution $g: Y
\map X$ in {\ref{subsection:pro-1}} is a non-trivial birational
contraction.

Put (formally) $H := v$ in {\ref{subsection:pro-1}}. Then it
follows that some of $a_i \ne 0$, since $f$ is mild,
and
$$
(f^*(v)^2) = (h^*(v)^2) - E^2 > (\deg f)(v^2).
$$
On the other hand, we have $(f^*(v)^2) = (\lambda(f)^2)(v^2)$, which contradicts the initial equality $\lambda(f)^2 = \deg f$. This proves
Theorem~\ref{theorem:main}.

\bigskip

\section{Examples and discussion}
\label{section:exam}

\refstepcounter{equation}
\subsection{}
\label{subsection:exam-1}

Let again the setting be as in {\ref{subsection:p-1}} above. With
the map $f$ are associated two kinds of (\emph{algebraic})
``entropies'': $\log\lambda(f)$ and $\log\deg f$. Both numbers
measure certain rate of volume growth of algebraic cycles on $X$.
For $\log\lambda(f)$, we have seen this in
Proposition~\ref{theorem:lambda-1-is-rho} (see also the initial
definition of $\lambda(f)$), whereas $\log\deg f$ can be described
as follows (cf. \cite[{\bf 1.1}]{Ilya-Ilya}). Let $\omega$ be a
positive closed $(1,1)$-form on $X$ (e.g. the restriction of the
Fubini--Studi form). Then
$$
\log\deg f = \lim_{n \to \infty}\frac{1}{n}\log \int_X
f^{n*}(\omega \wedge \omega).
$$
Note that the usual (\emph{topological}) entropy $h(f)$ of $f$
satisfies $h(f) \le \max\{\log\lambda(f),\log\deg f\}$ (see
\cite[Th\'eor\`eme 1]{Dinh-Sib}).

We proceed with specific examples.

\begin{ex}[{see \cite[Example 1.6]{Ilya-Ilya}}]
\label{example:surj-f-p-2} The map $f = [f_0:f_1:f_2]: \p^2
\dashrightarrow \p^2$ with quadratic components $$f_0 := xz +
y^2,\ f_1 := yz + x^2,\ f_2 := x^2 + y^2$$ has $\lambda(f) = 2$
and $\deg f = 3$. Thus $h(f) \le \log 3$.
\end{ex}

\begin{ex}[{see \cite[Proposition 3.2]{Ilya-ML}}]
\label{example:surj-cubic-p-2} The map $f = [f_0:f_1:f_2]: \p^2
\dashrightarrow \p^2$ with cubic components $$f_0 := x^2y + y^2z,\
f_1 := xyz,\ f_2 := x^2y + xy^2 + 2y^2z + z^2(x + y)$$ has
$\lambda(f) = 3$ and $\deg f = 4$. Thus $h(f) \le \log 4$.
\end{ex}

\refstepcounter{equation}
\subsection{}
\label{subsection:exam-2}

Let us give the following:

\begin{definition}
\label{definition:def-rank-1} The map $f: X \dashrightarrow X$ is
said to be \emph{of rank $1$} if $f^*$ has rank $1$.
\end{definition}

The preceding examples are particular instances of rank $1$ maps.
Furthermore, according to \cite[Theorem 1.6]{Forst} there is a
surjective map $\varphi: \p^2 \dashrightarrow X$, as is actually
the case for \emph{any rational surface} $X$. Then, composing this
$\varphi$ with a linear projection $\psi: X \map \p^2$, yields a
surjective endomorphism $f = \varphi \circ \psi$ of $X$. By
construction the linear map $f^*$ has rank $1$ because it factors
through $\varphi^*: \rea^2 \to \rea$.

\begin{ex}[{see \cite[Remark 2.3]{Ilya-dP}}]
\label{example:surj-cubic-p-1-p-1} Consider $X := \p^1 \times
\p^1$ equipped with product coordinates $[t_0:t_1] \times
[w_0:w_1]$. Then Segre embedding maps $X$ (isomorphically) onto a
quadric in $\p^3$ given by the equation $x_0x_3 - x_1x_2 = 0$ in
projective coordinates $x_0 := t_0w_0, x_1 := t_0w_1$, etc. Define
(surjective) map $\varphi: \p^2 \dashrightarrow X$ by
$$
\varphi([x:y:z]) = [xy:x^2-y^2] \times [xz:x^2-z^2]
$$
and take $\psi: X \map \p^2$ to be the projection from the point
$[1:0:0:1] \in \p^3$. Then we have
$$
f([t_0:t_1] \times [w_0:w_1]) = [t_0t_1w_0w_1:t_0^2w_1^2 -
t_1^2w_0^2] \times [t_0w_1(t_0w_0 - t_1w_1):t_0^2w_1^2 - (t_0w_0 -
t_1w_1)^2].
$$
The matrix of the operator $f^*$ --- in the basis as in
Remark~\ref{remark:nef-not-big} --- is $\begin{pmatrix} 2 & 2 \\ 2
& 2
\end{pmatrix}$. It follows that $\lambda(f) = 4$. Note also that
$\deg f = 8$. Thus $h(f) \le \log 8$.
\end{ex}

\refstepcounter{equation}
\subsection{}
\label{subsection:exam-3}

Every rank $1$ surjective map $f$ satisfies $\lambda(f) = $ the
trace of $f^*$ (cf. Proposition~\ref{theorem:lambda-1-is-rho}).
Hence we have $\lambda(f) \in \cel$. At the same time, as
Remark~\ref{remark:nef-not-big} shows, the property $\lambda(f)
\in \cel$ \emph{does not} characterize rank $1$ endomorphisms.
This abundance of ``degenerate'' surjective endomorphisms for del
Pezzo surfaces should be contrasted with the more restrictive
\emph{non-degenerate} case (see \cite[Theorem 1.2]{Ilya-dP}).

In order to get a grasp of rank $1$ maps one may wonder: does any
such $f$ factor through $\p^2$ via some surjective maps $\varphi$
and $\psi$ as above? It turns out that it \emph{need not}. To see
this let us modify the map $f$ from
Example~\ref{example:surj-cubic-p-1-p-1} as follows:
$$
f_\varepsilon([t_0:t_1] \times [w_0:w_1]) =
[t_0t_1w_0w_1:t_0^2w_1^2 - \varepsilon t_1^2w_0^2] \times
[t_0w_1(t_0w_0 - \varepsilon t_1w_1):t_0^2w_1^2 - (t_0w_0 -
t_1w_1)^2]
$$
for the parameter $\varepsilon$ close to $1$. In other words, we
are considering a family of rational maps $f_\varepsilon:
\p^1_{{\bf t}} \times \p^1_{{\bf w}} \dashrightarrow \p^1_{{\bf
t}} \times \p^1_{{\bf w}}$, written shortly as
$$
{\bf t} \times {\bf w} \mapsto f_\varepsilon({\bf t}) \times
f_\varepsilon({\bf w})
$$
(compare with \cite[{\bf 1.3}]{Ilya-Ilya}).

The indeterminacy locus $Z_{f_\varepsilon}$ of $f_\varepsilon$
consists of two components $(f_\varepsilon({\bf t}) = 0)$ and
$(f_\varepsilon({\bf w}) = 0)$. It follows that
$Z_{f_\varepsilon}$ is \emph{finite}, of total length $16$, for
all $\varepsilon$. Hence the schemes $Z_{f_\varepsilon}$ form a
\emph{flat family}. Now, since $f_1 = f$ is surjective, the proof
of \cite[Theorem 1.10]{Ilya-Ilya} applies verbatim and gives
surjectivity of the general map $f_\varepsilon$ as well. But for
the latter following holds:

\begin{prop}
\label{theorem:no-fact} If $\varepsilon \ne 1$ is generic, then we
have $f_\varepsilon \ne \varphi \circ \psi$ for any surjective
maps $\varphi, \psi$ as above.
\end{prop}

\begin{proof}
Suppose the contrary. Then, since $f^*_\varepsilon = f^*_1$ for
all $\varepsilon$ (cf. Example~\ref{example:surj-cubic-p-1-p-1}),
it follows that $\psi: X \map \p^2$ must necessarily be a
\emph{regular} projection. Hence there is an involution $\iota$,
acting biregularly on $X$, such that $X \slash \iota = \p^2$ and
$f_\varepsilon \circ \iota = f_\varepsilon$.

Put $t := t_0 \slash t_1$ and $w := w_0 \slash w_1$. It is easy to
see that $\iota$ acts as follows:
$$
(t,w) \mapsto (A(w),A^{-1}(t))
$$
for some $A \in \pg_2(\com)$. We also have
$$
f_\varepsilon(t,w) = \big(\frac{tw}{t^2 - \varepsilon w^2},
\frac{t(tw - \varepsilon)}{t^2 - (tw - 1)^2}\big).
$$
The first component of $f_\varepsilon$ depends only on the ratio
$s := t \slash w$. Then the property $f_\varepsilon \circ \iota =
f_\varepsilon$ forces $s' := A(w) \slash A^{-1}(t)$ to satisfy
either $s' = s$ or $s' = -\varepsilon\slash s$. This, in turn,
gives the following options:

\begin{itemize}

    \item $A(w) = K \slash w$ for some $K \in \com$;

    \smallskip

    \item $A(w) = Lw$ for $L^2 = -\varepsilon$.

\end{itemize}

It is now immediate that the second component
$\displaystyle\frac{t(tw - \varepsilon)}{t^2 - (tw - 1)^2}$ of
$f_\varepsilon$ fails to be $\iota$-invariant. This contradiction
finishes the proof of Proposition~\ref{theorem:no-fact}.
\end{proof}

Thus we see that even for ``simple'' surfaces the structure of
their surjective endomorphisms is quite tricky. It would be
interesting to obtain a rough, at least, description of such
endomorphisms $f$ --- perhaps in terms of a ``statistics'' of
numbers $\lambda(f) \slash \deg f$ (compare with \cite[Lecture
3]{Arnold-dubna}).

\bigskip

\bigskip

{\bf Competing Interest declaration}. There is none to declare.

\bigskip

\bigskip

\thanks{{\bf Acknowledgments.} I am grateful to I. Zhdanovskiy for
fruitful conversations. The work was carried out at the Center for
Pure Mathematics (MIPT) and was supported by the Russian Science
Foundation under grant \textnumero\, 25-21-00083
(https://rscf.ru/project/25-21-00083/).

\end{document}